\documentclass[11pt]{amsart}

\usepackage{amsmath, amsthm, amssymb}

\usepackage{parskip}
  
\usepackage{lmodern}
\usepackage{palatino}                       
\usepackage[bigdelims,vvarbb]{newpxmath}    
\usepackage[scaled=0.95]{inconsolata}       
\usepackage[T1]{fontenc}                    

\usepackage{comment}

\usepackage[abs]{overpic}		  
\usepackage[all,cmtip]{xy}

\usepackage{xcolor}
\definecolor{indigo}{rgb}{0.29, 0.0, 0.51}  
\usepackage[colorlinks, urlcolor=indigo, linkcolor=indigo, citecolor=indigo]{hyperref}

\usepackage[hcentering, vcentering, total={5.9in, 8.2in}]{geometry}  

\usepackage{tikz}
\usetikzlibrary{trees}

\usepackage{mathtools}

\theoremstyle{plain}
\newtheorem{theorem}{Theorem}
\newtheorem{corollary}[theorem]{Corollary}
\newtheorem{proposition}[theorem]{Proposition}
\newtheorem{lemma}[theorem]{Lemma}

\newtheorem{problem}[theorem]{Problem}

\theoremstyle{definition}
\newtheorem{definition}[theorem]{Definition}

\theoremstyle{remark}
\newtheorem{remark}[theorem]{Remark}

\numberwithin{theorem}{section}

\newcommand{\R}{\mathbb{R}}           
\newcommand{\Z}{\mathbb{Z}}           
\newcommand{\NS}{{\mathbb{S}}}

\newcommand{\D}{{\mathbb{D}}}

\newcommand{\Op}{{\mathcal{O}p}}

\renewcommand{\L}{\mathcal{L}}           
\newcommand{\FL}{\mathcal{FL}}           

\makeatletter
\newcommand*\bigcdot{\mathpalette\bigcdot@{0.6}}
\newcommand*\bigcdot@[2]{\mathbin{\vcenter{\hbox{\scalebox{#2}{$\m@th#1\bullet$}}}}}
\makeatother

\DeclareFontFamily{U} {cmr}{}
\DeclareFontShape{U}{cmr}{m}{n}{
  <-6> cmr5
  <6-7> cmr6
  <7-8> cmr7
  <8-9> cmr8
  <9-10> cmr9
  <10-12> cmr8
  <12-> cmr9}{}
\DeclareSymbolFont{Xcmr} {U} {cmr}{m}{n}
\DeclareMathSymbol{\Phi}{\mathord}{Xcmr}{8}

\newcommand{\Diff}{\operatorname{Diff}}

\newcommand{\rel}{\operatorname{rel}}

\newcommand{\Maps}{\operatorname{Maps}}

\newcommand{\Id}{{\operatorname{Id}}}

\newcommand{\rot}{\operatorname{rot}}
\newcommand{\tb}{\operatorname{tb}}

\newcommand{\Cont}{\operatorname{Cont}}
\newcommand{\FCont}{\operatorname{FCont}}

\newcommand{\CFr}{\operatorname{CFr}}

\newcommand{\Germs}{\operatorname{Germs}}

\newcommand{\pr}{\operatorname{pr}}

\newcommand{\image}{\operatorname{Image}}

\newcommand{\C}{\operatorname{C}}
\newcommand{\FC}{\operatorname{FC}}
\newcommand{\OT}{\operatorname{OT}}

\newcommand{\Emb}{{\operatorname{Emb}}}
\newcommand{\FEmb}{{\operatorname{FEmb}}}
\newcommand{\FDiff}{{\operatorname{FDiff}}}
\newcommand{\F}{{\operatorname{\mathcal{F}}}}
\newcommand{\Germ}{\operatorname{Germ}}

\begin{document} 

\title{Strongly overtwisted contact $3$-manifolds}

\author{Eduardo Fern\'{a}ndez}
\address{Department of Mathematics\\ University of Georgia\\ Athens\\ GA, USA}
\email{eduardofernandez@uga.edu}


\maketitle
{\centering\footnotesize \emph{Dedicated to Francisco Presas}\par}

\begin{abstract}
We prove the existence of a subclass of overtwisted contact structures, called strongly overtwisted, on a $3$-manifold that satisfy a complete $h$-principle without prescribing the contact structures over any subset of the $3$-manifold. As a consequence, the homotopy type of the space of overtwisted disk embeddings into a strongly overtwisted contact $3$-manifold is determined. A complete $h$-principle for a subclass of loose Legendrians is also derived from the main result. In general, the method allows us to deduce an $h$-principle for overtwisted disks that are fixed near the boundary in an arbitrary overtwisted contact $3$-manifold.
\end{abstract}


\section{Introduction}\label{sec:intro}

In recent years, several classification results and gluing type theorems have appeared for families of tight contact structures on $3$-manifolds \cite{EliashbergMishachev:Tight, FMP:Tight, FMin:Cables, FM:Dehn}. On the contrary, our understanding of \em families \em of overtwisted contact structures on $3$-manifolds has remained limited. The reason is that the $h$-principle of Eliashberg \cite{Eliashberg:OT} is parametric only when an overtwisted disk is assumed to be \em fixed. \em In fact, prior this article, the only known classification result for families of overtwisted contact structures was for $\NS^1$-families of overtwisted contact structures on $\NS^3$; due, independently, to Chekanov and Vogel \cite{Vogel:OTDisk}; and implies that \em rigidity \em could arise when studying families of overtwisted contact structures. In particular, the condition about the fixed overtwisted disk in \cite{Eliashberg:OT} cannot be removed in general. Therefore, obtaining a description of the homotopy type of the space of contact structures isotopic to a fixed overtwisted contact structure is an open problem since the appearence of \cite{Eliashberg:OT}. An equivalent problem appeared in the Problem List on Low Dimensional Contact Topology from the 2001 International Georgia Topology Confence 
\begin{problem}[Problem 31 from \cite{ContactProblemList}] \label{problem}
    Determine the group $\Cont(M,\xi)$ for $(M,\xi)$ an overtwisted $3$-manifold. 
\end{problem}

In this article, we define an infinite subclass of overtwisted contact $3$-manifolds, which we call \em strongly overtwisted. \em This subclass satisfies a complete $h$-principle without prescribing the contact structures over any subset of the $3$-manifolds. This allows us to describe the homotopy type of the space of strongly overtwisted contact structures on a fixed $3$-manifold. It also implies a complete $h$-principle for the contactomorphism group of a strongly overtwisted contact $3$-manifold, solving Problem \ref{problem} for every strongly overtwisted contact $3$-manifold. 

For a general overtwisted contact $3$-manifold, the method allows us to prove a parametric $h$-principle for overtwisted contact structures that coincide near the \em boundary \em of an overtwisted disk, relaxing the fixed overtwisted disk condition from \cite{Eliashberg:OT}. In particular, a parametric $C^0$-dense $h$-principle for overtwisted disks \em fixed near the boundary \em in any overtwisted contact $3$-manifold is established. 

\subsection{Main results}

Throughout this article, every $3$-manifold $M$ is oriented, and every contact structure $\xi$ on $M$ is assumed to be positive and co-oriented. 

\subsubsection{Strongly overtwisted contact $3$-manifolds}

Define the \em negative part of the Euler characteristic \em of a closed and orientable surface $\Sigma$ to be $$\chi_{-}(\Sigma)=\max(0,-\chi(\Sigma)).$$

\begin{definition} \noindent
    \begin{itemize} 
    \item [(i)] A closed, connected and orientable embedded surface $\Sigma\subseteq (M,\xi)$ is said to be a \em strongly overtwisted surface \em if 
     $|e(\xi)[\Sigma]|>\chi_{-} (\Sigma). $
    \item [(ii)] A compact, connected and orientable properly embedded surface $\Sigma\subseteq (M,\xi)$ that is convex near the connected Legendrian boundary $\partial \Sigma\subseteq \partial M$ is said to be \em strongly overtwisted \em if $\tb(\partial \Sigma)\leq 0$ and $\tb(\partial\Sigma)+|\rot_{\Sigma}(\partial \Sigma)|>-\chi(\Sigma)$.
    \item [(iii)] A contact $3$-manifold $(M,\xi)$ is said to be \em strongly overtwisted \em if it contains an embedded strongly overtwisted surface $\Sigma\subseteq (M,\xi)$. The surface $\Sigma$ is required to be properly embedded if $\partial \Sigma\neq \emptyset$. 
    \end{itemize}
\end{definition}

If $(M,\xi)$ is a closed strongly overtwisted contact $3$-manifold, the Euler class $e(\xi)\in H^2(M,\Z)$ is non torsion. Therefore, neither $\NS^3$ nor all rational homology $3$-spheres admit a strongly overtwisted contact structure. 

On the other hand, the class of strongly overtwisted contact $3$-manifolds is stable under contact surgery along isotropic submanifolds with strongly overtwisted complement. For instance, the contact connected sum of a contact manifold with a strongly overtwisted contact manifold is strongly overtwisted. 

Natural examples of strongly overtwisted contact $3$-manifolds are $\NS^1$-invariant neighborhoods of convex strongly overtwisted surfaces. More abstractly, if $(M,\eta)$ is a closed $3$-manifold $M$ equipped with a co-oriented plane field $\eta$, and $\Sigma\subseteq M$ is an embedded, connected and oriented surface such that $|e(\eta)[\Sigma]|>\chi_{-}(\Sigma)$; then the $h$-principle for overtwisted contact structures \cite{Eliashberg:OT} implies the existence of a strongly overtwisted contact structure $\xi$ on $M$, which is homotopic to $\eta$ as a plane field. In other words, strongly overtwisted contact structures exist abundantly as long as there are not obvious homotopical obstructions.

The motivation behind the previous definition rests on the following
\begin{theorem}[Bennequin-Eliashberg Inequalities]\label{thm:EliashbergBound}
    Every open neighborhood of a strongly overtwisted surface is overtwisted. In particular, every strongly overtwisted contact $3$-manifold is overtwisted.
\end{theorem}
The case where $\Sigma$ is closed was proved by Eliashberg in \cite{Eliashberg:Twenty}. The non-empty boundary case follows from the Bennequin-Eliashberg inequality \cite{Bennequin,Eliashberg:Twenty} coupled with convex surface theory \cite{Giroux:Convexity,Giroux:Tightness}. Here, we are making use of the condition about the Thurston-Bennequin number in a crucial way. Indeed, there are non-loose Legendrians that violate the Bennequin-Eliashberg inequality \cite{EliashbergFraser}. See Theorem \ref{thm:OTatInfinity} for further details.

\begin{remark}
 It follows from Theorem \ref{thm:EliashbergBound} that the Euler class of a strongly overtwisted contact structure on a closed $3$-manifold $M$ cannot be realized as the Euler class of a tight contact structure. Even more, only finitely many elements of $H^2(M,\Z)$ (that are always even) can be realized as the Euler class of a \em not \em strongly overtwisted contact structure \cite{Eliashberg:Twenty}.
\end{remark}

\subsubsection{The $h$-principle for strongly overtwisted contact structures}

Let $(M,\xi)$ be a compact contact $3$-manifold. A formal contact structure on $M$ is a co-oriented plane field on $M$. If $M$ has non-empty boundary, we require that every formal contact structure agrees with $\xi$ near the boundary. We denote by $\FC(M;\xi)$ the component of the space of such formal contact structures on $M$ that contains $\xi$. Similarly, $\C(M;\xi)\subseteq \FC(M;\xi)$ stands for the subspace of contact structures. Our main result is a complete $h$-principle for strongly overtwisted contact structures:

\begin{theorem}\label{thm:Main}
    Let $(M,\xi)$ be a compact strongly overtwisted contact $3$-manifold. The inclusion 
        \[ i_{\C}:C(M;\xi)\hookrightarrow \FC(M;\xi)\] is a weak homotopy equivalence.
\end{theorem}

\begin{remark}
The difference between the main result in \cite{Eliashberg:OT}, see Theorem \ref{thm:EliashbergOT} below, and Theorem \ref{thm:Main} is that the restriction of being fixed near an overtwisted disk is dropped. In particular, contrary to \cite{Vogel:OTDisk}, rigidity for families of strongly overtwisted contact structres does not appear.
\end{remark}
\begin{remark}
 After fixing a trivialization of $TM\cong M\times \R^3$, the assignment 
\[ \FC(M,\xi)\rightarrow \Maps_{\xi}(M,\NS^2),\hat{\xi}\mapsto f_{\hat{\xi}}; \]
which sends every contact structure to its Gauss map is a homotopy equivalence. Here, the subscript $\xi$ on the space of maps denotes the path-connected component of the Gauss map $f_{\xi}$ in the space $\Maps(M,\NS^2)$ of maps from $M$ to $\NS^2$ that coincide with $f_{\xi}$ near $\partial M$. Therefore, Theorem \ref{thm:Main} reduces the study of families of strongly overtwisted contact structures to a problem that can be treated by obstruction theoretic means. 
\end{remark}

\subsubsection{The $h$-principle for the contactomorphism group of a strongly overtwisted contact $3$-manifold}
We will deduce Theorem \ref{thm:Main} from the analog $h$-principle for contactomorphisms which solves Problem \ref{problem} for every strongly overtwisted contact $3$-manifold. In the result below, $\FCont(M,\xi)$ denotes the group of formal contactomorphisms (see Definition \ref{def:FormalContactomorphism}) that coincides with the identity near $\partial M$ and $\Cont(M,\xi)\subseteq \FCont(M,\xi)$ denotes the contactomorphism group of $(M,\xi)$. 

\begin{theorem}\label{thm:MainContactomorphisms}
Let $(M,\xi)$ be a compact strongly overtwisted contact $3$-manifold. Then, the natural inclusion \[ i_{\Cont}:\Cont(M,\xi)\hookrightarrow \FCont(M,\xi) \]
is a weak homotopy equivalence.
\end{theorem}

Theorems \ref{thm:Main} and \ref{thm:MainContactomorphisms} both follow by combining the $h$-principle for overtwisted contact $3$-manifolds \cite{Eliashberg:OT} with a new parametric and $C^0$-dense $h$-principle for strongly overtwisted surfaces (Theorem \ref{thm:H-PrincipleOvertwistedSurfaces}). This $h$-principle heavily uses Theorem \ref{thm:EliashbergBound}, which is a manifestation of Giroux Tightness Criteria \cite{Giroux:Tightness}. Although there are \em algebraic \em Giroux Criteria in higher dimensions \cite{Avdek}, it remains an open problem to find a higher-dimensional analog of the Giroux Criteria. On the other hand, there are no Bennequin-Eliashberg type inequalities in higher dimensions \cite{CasalsPancholiPresas,HondaHuang,Murphy:Loose}. Therefore, a priori, it seems unlikely that an approach similar to the one followed in this article will work to study families of overtwisted contact structures in higher dimensions \cite{BEM}.

\subsection{Applications}

\subsubsection{General overtwisted contact $3$-manifolds}

Let $(M,\xi)$ be an overtwisted contact $3$-manifold and $\D^2_{\OT}\subseteq (M,\xi)$ an overtwisted disk. Observe that the contact manifold $(M\backslash \Op(\partial\D^2_{\OT}),\xi)$ is strongly overtwisted.

Denote by $\FCont(M;\xi,\rel \partial \D^2_{\OT})\subseteq \FC(M;\xi)$ the subspace of formal contact structures that coincide with $\xi$ near $\partial \D^2_{\OT}$, and by $\C(M;\xi,\rel \partial \D^2_{\OT})=\C(M;\xi)\cap \FCont(M;\xi,\rel \partial \D^2_{\OT})$ the corresponding subspace of contact structures. Theorem \ref{thm:Main} directly implies the following

\begin{corollary}\label{cor:EliashbergSoft}
    Let $(M,\xi)$ be a compact overtwisted contact $3$-manifold and $\D^2_{\OT}\subseteq (M,\xi)$ an overtwisted disk. Then, the inclusion 
    \[ i_{\C}:\C(M;\xi,\rel \partial\D^2_{\OT})\hookrightarrow \FC(M;\xi;\rel \partial \D^2_{\OT}) \]
    is a weak homotopy equivalence.
\end{corollary}

\begin{remark}
The difference between Corollary \ref{cor:EliashbergSoft} and the $h$-principle in \cite{Eliashberg:OT} is that the fixed overtwisted disk requirement has been relaxed: it is enough to fix the boundary of an overtwisted disk to obtain flexibility.
\end{remark}

Corollary \ref{cor:EliashbergSoft} has a corresponding result for overtwisted disks. Let $i:\D^2_{\OT}\hookrightarrow (M,\xi)$ be the inclusion of some overtwisted disk. Denote by $\FEmb((\Op(\D^2_{\OT}),\xi_{\OT}),(M,\xi))$ the space of formal isocontact embeddings $(\Op(\D^2_{\OT}),\xi_{\OT})\rightarrow (M,\xi)$ that coincide with $i$ near $\partial \D^2_{\OT}$, and let $\Emb_{\partial}((\Op(\D^2_{\OT},\xi_{\OT}),(M,\xi))$ be the subspace of isocontact embeddings. We refer the reader to \ref{subsec:Germs} for more details.

The $h$-principle for strongly overtwisted surfaces (Theorem \ref{thm:H-PrincipleOvertwistedSurfaces}) automatically implies

\begin{theorem}\label{thm:OTFixedBoundary}
The inclusion \[\Emb_{\partial}(\Op(\D^2_{\OT}),\xi_{\OT}),(M,\xi))\hookrightarrow \FEmb_{\partial}((\Op(\D^2_{\OT}),\xi_{\OT}),(M,\xi))\] is a $C^0$-dense weak homotopy equivalence. 
\end{theorem}
\begin{remark}
    The $C^0$-dense weak homotopy equivalence just says that the previous $h$-principle is $C^0$-dense \cite{EliashbergMishachev:Book}. See Theorem \ref{thm:microfibration} for the precise definition.
\end{remark}

Observe that, in particular, the previous result allows displacing families of embeddings away from an overtwisted disk $\D^2_{\OT}$  with contact isotopies, as long as they do not intersect the boundary $\partial \D^2_{\OT}$ and there is no smooth obstruction.

\subsubsection{Overtwisted disks in strongly overtwisted contact $3$-manifolds.}

Let $\Emb^{\OT}(\D^2,(M,\xi))$ be the space of embeddings of overtwisted disks into a contact $3$-manifold $(M,\xi)$. The space of contact frames of $(M,\xi)$ is denoted by $\CFr(M,\xi)$, this is the space of pairs $(p,(v_1,v_2))$ where $p\in M$ and $(v_1,v_2)$ is an oriented basis of $\xi_p$.

\begin{corollary}\label{cor:OvertwistedDisks}
Let $(M,\xi)$ be a strongly overtwisted contact $3$-manifold.
\begin{itemize}
    \item [(i)] For every pair of overtwisted disks $\D^2_{\OT,1},\D^2_{\OT,2}\subseteq (M,\xi)$ there exists a contact isotopy $\varphi_t\in \Cont(M,\xi)$, $t\in[0,1]$, such that $\varphi_1(\D^2_{\OT,1})=\D^2_{\OT,2}$.
    \item [(ii)] The map \[ \Emb^{\OT}(\D^2,(M,\xi))\rightarrow \CFr(M,\xi), \] induced by the evaluation of the $1$-jet at the origin, is a weak homotopy equivalence. 
\end{itemize}
\end{corollary}

\subsubsection{Legendrians in strongly overtwisted manifolds.}

Recall that a Legendrian $L\subseteq (M,\xi)$ is said to be \em loose \em if there exists an overtwisted disk $\D^2_{\OT}$ such that $\D^2_{\OT}\cap L=\emptyset$. We will say that a Legendrian $L\subseteq (M,\xi)$ is \em strongly loose \em if there exists a strongly overtwisted surface $\Sigma\subseteq (M,\xi)$ such that $\Sigma\cap L=\emptyset$.  

Notice that the strongly looseness property just depends on the smooth link type. Therefore, we will say that a smooth link type $K$ in $M$ is strongly loose in $(M,\xi)$, if there exists some smooth representative of $K$ in the complement of a strongly overtwisted surface. 

Given a link type $K$ we will denote by $\L(K,(M,\xi))$ the space the of Legendrian embeddings realizing the link type $K$ and by $\FL(K,(M,\xi))$ the corresponding space of formal Legendrian embeddings. 

\begin{corollary}\label{cor:StronglyLooseLegendrians}
Let$(M,\xi)$ be a strongly overtwisted contact $3$-manifold. 
\begin{itemize}
  \item [(i)] Let $L_1, L_2\subseteq (M,\xi)$ be two loose Legendrians which are formally Legendrian isotopic. Then, $L_1$ and $L_2$ are Legendrian isotopic. 
  \item [(ii)] Let $K$ be a strongly loose smooth link type in $(M,\xi)$. Then, the inclusion 
  \[  i_{\L}:\L(K,(M,\xi))\hookrightarrow \FL(K,(M,\xi)) \]
  is a weak homotopy equivalence.
\end{itemize}
\end{corollary}

As shown by Etnyre \cite{Etnyre:LegendriansOT}, see also \cite{EliashbergFraser}, two formally isotopic loose Legendrians are coarsely equivalent. Moreover, if there is a common overtwisted disk in the complement of both Legendrians then these are Legendrian isotopic by the work of Dymara \cite{Dymara:LegendriansOT}. This condition can be relaxed: every formal Legendrian isotopy $L_t$, $t\in[0,1]$, connecting two genuine Legendrians $L_0$ and $L_1$ is formally isotopic, $\rel\{0,1\}$, to a Legendrian isotopy if there exists an isotopy of overtwisted disks $\Delta_t$, $t\in [0,1]$, such that $L_t\cap \Delta_t=\emptyset$ for all $t\in[0,1]$. This last condition cannot be removed in general \cite{Vogel:OTDisk}. More generally, Cardona and Presas \cite{CardonaPresas} established an $h$-principle for Legendrians with a fixed overtwisted disk in the complement. All these flexibility results build on \cite{Eliashberg:OT}. From this point of view Corollary \ref{cor:StronglyLooseLegendrians} is a natural consequence of our main results. 

\begin{remark}
    Corollary \ref{cor:OvertwistedDisks} and \ref{cor:StronglyLooseLegendrians} are two instances of the same phenomena. In general, given a submanifold $N\subseteq (M,\xi)$ we say that the smooth isotopy class of $N$ is strongly loose if every component of $M\backslash N$ contains a strongly overtwisted surface. Then, every subspace of the space of smooth embeddings of $N$ into $M$ smoothly isotopic to the inclusion for which the contact isotopy extension theorem applies is governed by an $h$-principle. For instance, this applies for transverse embeddings, codimension $0$ isocontact embeddings, embeddings of surfaces with fixed characteristic foliation, etc. 
\end{remark}

\subsection{Outline}
The article is organized as follows. In Section \ref{sec:contactomorphisms} we introduce notation and review some consequences of \cite{Eliashberg:OT} regarding the contactomorphism group of a contact $3$-manifold at infinity. In Section \ref{sec:StronglyOvertwistedSurfaces} we review the Microfibration Trick from \cite{FMP:Tight} and we prove the $h$-principle for strongly overtwisted surfaces with fixed characteristic foliation (Theorem \ref{thm:H-PrincipleOvertwistedSurfaces}). In Section \ref{sec:proofs} we prove the main results: Theorems \ref{thm:Main} and \ref{thm:MainContactomorphisms}. Finally, Corollaries \ref{cor:OvertwistedDisks} and \ref{cor:StronglyLooseLegendrians} are proved in Section \ref{sec:Applications}

\subsection*{Acknowledgements} The author is very grateful to Fran Presas, who taught him a significant amount of mathematics, including $h$-principles and symplectic topology, and for countless discussions about flexibility in contact topology. The author would also like to acknowledge Fabio Gironella, \'Alvaro de Pino, and Guillermo S\'anchez for valuable discussions and feedback on this work, as well as David Gay and Gordana Mati\'c for their interest in this project and their constant support. The author thanks the anonymous referee for carefully reading the manuscript.

\section{Contactomorphisms of contact $3$-manifols overtwisted at infinity}\label{sec:contactomorphisms}

In this Section, we introduce several notations and conventions that we will follow throughout the article. We also recall the main result of \cite{Eliashberg:OT}. Finally, we explain how to use \cite{Eliashberg:OT} to deduce an $h$-principle for the contactomorphism group of a contact $3$-manifold overtwisted at infinity (see Proposition \ref{prop:ContactomorphismsOTAtInfinity}).

\subsection{Formal contact structures}

\begin{definition}
    Let $M$ be a oriented $3$-manifold. A \em formal contact structure \em on $M$ is a co-oriented plane field $\xi\subseteq TM$. The pair $(M,\xi)$ is a \em formal contact $3$-manifold\em.
\end{definition}

The space of formal contact structures on $M$ is denoted by $\FC(M)$. In the case that $M$ is open or has non-empty boundary we will always assume that every formal contact structure coincides with an unspecified but \em fixed \em contact structure outside a compact set or near the boundary. Such a contact structure would be clear from the context. Given a formal contact structure $\xi\in \FC(M)$, we will denote the path connected component of $\FC(M)$ containing $\xi$ by $\FC(M;\xi)$. 

We will denote by $\C(M)\subseteq \FC(M)$ the subspace of contact structures on $M$. Similarly, for a given contact structure $\xi\in \C(M)$ we will denote by $\C(M;\xi)$ the connected component of $\C(M)$ containing $\xi$. 

If $A\subseteq M$ is a closed subset equipped with some germ of contact structure we will also consider the subspaces $\FC(M,\rel A)$ and $\C(M,\rel A)$ of formal and genuine contact structures that induce the same germ over $A$. Similarly, $\FC(M;\xi,\rel A)$ and $\C(M;\xi,\rel A)$ will denote the components containing $\xi$.

\subsection{Formal contactomorphisms}
A \em formal diffeomorphism \em is a pair $(\varphi, F_s)$ such that $\varphi\in \Diff(M)$, is an orientation preserving diffeomorphism; and $F_s:TM\rightarrow TM$, $s\in[0,1]$, is a homotopy of vector bundle isomorphisms covering $\varphi$ such that $F_0=d\varphi$. The space of formal diffeomorphisms is denoted by $\FDiff(M)$. Note that the natural inclusion $ \Diff(M)\hookrightarrow \FDiff(M) $ is a homotopy equivalence. In the case that $M$ is open we assume that every (formal) diffeomorphism has compact support. Similarly, if $\partial M\neq \emptyset$ we will assume that every (formal) diffeomorphism is the identity near $\partial M$. The connected component of the identity $(\Id,d\Id)$ in $\FDiff(M)$ is denoted by $\FDiff_0(M)$; and in $\Diff(M)$ by $\Diff_0(M)=\Diff(M)\cap \FDiff_0(M)$. 

\begin{definition}\label{def:FormalContactomorphism}
    A \em formal contactomorphism \em of a formal contact $3$-manifold $(M,\xi)$ is a formal diffeomorphism $(\varphi,F_s)\in \FDiff(M)$ such that $F_1(\xi)=\xi$ as co-oriented plane fields.
\end{definition}

The group of formal contactomorphisms is denoted by $\FCont(M,\xi)$. If $A\subseteq M$ is a closed subset we will denote by $\FCont(M,\xi;\rel A)\subseteq \FCont(M,\xi)$ the subgroup conformed by those formal contactomorphisms that agree with the identity $\Id=(\Id,d\Id)$ over $\Op(A)$. 

In the case that $(M,\xi)$ is a contact $3$-manifold, we will say that a formal contactomorphism $(\varphi,F_s)\in\FCont(M,\xi)$ is a (genuine) \em contactomorphism \em when $F_s=d\varphi$. In this case we will simply write $\varphi$ instead of $(\varphi,d\varphi)$. The group of contactomorphisms of $(M,\xi)$ is denoted by $\Cont(M,\xi)$.

The homotopy type of the space of (formal) contact structures and (formal) contactomorphisms are closely related: 
\begin{lemma}[Gray Fibrations]\label{lem:Gray} \hfill
  \begin{itemize} 
        \item [(i)] Let $(M,\xi)$ be a contact $3$-manifold. The map $$G:\Diff_0(M)\rightarrow \C(M;\xi),\varphi\mapsto \varphi_* \xi;$$ is a fibration with fiber $\Cont_0(M,\xi)=\Diff_0(M)\cap \Cont(M,\xi)$.
        \item [(ii)] Let $(M,\xi)$ be a formal contact $3$-manifold. The map $$FG:\FDiff_0(M)\rightarrow \FC(M;\xi),(\varphi,F_s)\mapsto F_1(\xi);$$ is a fibration with fiber $\FCont_0(M,\xi)=\FDiff_0(M)\cap \FCont(M,\xi)$.
    \end{itemize}
    Both (i) and (ii) also hold relative to a closed subset $A\subseteq M$. 
\end{lemma}

\subsection{Eliashberg's overtwisted $h$-principle}

The importance of overtwisted contact structures relies on the following $h$-principle due to Eliashberg, later generalized to all dimensions by Borman, Eliashberg and Murphy \cite{BEM}.

\begin{theorem}[Eliashberg \cite{Eliashberg:OT}]\label{thm:EliashbergOT}
Let $M$ be a $3$-manifold and $\D^2_{\OT}\subseteq M$ an embedded $2$-disk equipped with the germ of an overtwisted disk. Then, the natural inclusion 
\[  i:\C(M,\rel \D^2_{\OT})\hookrightarrow \FC(M, \rel \D^2_{\OT}) \]
is a weak homotopy equivalence.
\end{theorem}

As mentioned before, the work \cite{Vogel:OTDisk} implies that the condition of being fixed near the overtwisted disk cannot be removed in general.

\subsection{Contactomorphisms of contact $3$-manifolds overtwisted at infinity}

\begin{definition}
    A contact $3$-manifold $(M,\xi)$ is \em overtwisted at infinity \em if for every compact set $K\subseteq (M,\xi)$ every component of $(M\backslash K, \xi)$ is overtwisted.
\end{definition}

The prototypical example of a contact $3$-manifold overtwisted at infinity is the complement of an overtwisted disk. For us, the following examples, which reformulate Theorem \ref{thm:EliashbergBound}, will be used in an instrumental way in the proofs of the main results of this article. 

\begin{theorem}(Bennequin-Eliashberg Inequalities)\label{thm:OTatInfinity}
    Let $(\Sigma\times (-1,1),\xi)\subseteq (M,\xi)$ be an open neighborhood of a strongly overtwisted surface $\Sigma=\Sigma\times\{0\}$. If $\partial \Sigma\neq \emptyset$ further assume that $(\Op(\partial \Sigma)\times (-1,1),\xi)$ is vertically invariant. Then, the contact manifold $(\Sigma\times(-1,1),\xi)$ is overtwisted at infinity.
\end{theorem}
\begin{proof}
    In the case that $\Sigma$ is closed the result follows from Eliashberg's bound on the Euler class of a tight contact structure \cite{Eliashberg:Twenty} since every surface $\Sigma\times\{t\}\subseteq (\Sigma\times(-1,1),\xi)$ is strongly overtwisted.

    Assume that $\partial \Sigma\neq \emptyset$. It is enough to see that every $\Sigma\times\{t\}$ can be smoothly isotoped, relative to the boundary, to a $C^\infty$-close surface $\widehat{\Sigma}_t$ that contains an overtwisted disk in every small open neighborhood. Since $\tb(\Sigma\times\{t\})\leq 0$, and $\Sigma$ is convex near the boundary, we can smoothly isotope, relative to the boundary, the surface $\Sigma\times\{t\}$ into a convex surface $\widehat{\Sigma}_t$ which is $C^\infty$-close to $\Sigma\times\{t\}$, see \cite{Giroux:Convexity,Honda:Classification}. Now, the Bennequin-Eliashberg inequality \cite{Bennequin,Eliashberg:Twenty} coupled with Giroux Tightness Criteria \cite{Giroux:Tightness} imply the existence of some overtwisted disk in a vertically invariant neighborhood of $\widehat{\Sigma}_t$. Indeed, if this was not the case, the Legendrian $\partial \Sigma$ would satisfy the Bennequin-Eliashberg inequality. 
\end{proof}

Now, we will determine the contactomorphism group of every contact $3$-manifold which is overtwisted at infinity. First, we prove the non-parametric $h$-principle:

\begin{lemma}\label{lem:Pi0Cont}
Let $(M,\xi)$ be a contact $3$-manifold overtwisted at infinity. Then, the inclusion $\Cont(M,\xi)\hookrightarrow \FCont(M,\xi)$ is surjective on path-components.
\end{lemma}
\begin{proof}
    Let $(\varphi,F_s)\in \FCont(M,\xi)$ be a formal contactomorphism. Consider the homotopy of formal contact structures $\xi_s=F_s(\xi)$, $s\in[0,1]$. Note that $\xi_0=\varphi_* \xi$ and $\xi_1=\xi$ are contact structures. By Theorem \ref{thm:EliashbergOT} there exists a homotopy of formal contact structures $\xi^t_s$, $(s,t)\in[0,1] \times [0,1]$, such that 
  \begin{itemize}
      \item [(i)] $\xi^t_s=\xi_s$ for $(s,t)\in [0,1]\times\{0\}\cup \{0,1\}\times[0,1]$.
      \item [(ii)] $\xi^1_s$ is contact. 
  \end{itemize}

  Hence, we can find a homotopy of formal contactomorphisms $(\varphi,F^t_s)\in \FCont(M,\xi)$, $(s,t)\in[0,1]\times[0,1]$, such that 
  \begin{itemize}
      \item [(a)] $F^t_s(\xi)=\xi^t_s$, for $(s,t)\in[0,1]\times[0,1]$. 
      \item [(b)] $F^t_s=F_s$ for $(s,t)\in [0,1]\times\{0\}\cup \{0,1\}\times[0,1]$.
  \end{itemize}
  Since $\xi^1_s$ is contact we can apply Gray Stability to find a smooth isotopy $H_s:M\rightarrow M$, $s\in[0,1]$, such that $H_1=\Id$ and $(H_s)_* \xi^1_s=\xi^1_1=\xi$. Extend the previous homotopy over $t\in[1,2]$ via the formula
  \[ (\varphi_t, F^t_s)=(H_{2-t}\circ \varphi, dH_{2-t} \circ F^1_{(2-t)s})\in \FCont(M,\xi). \]

  For $t=2$ we have that $(\varphi_2,F^2_s)=(H_0\circ \varphi, dH_0\circ F^1_0)=(H_0\circ \varphi, dH_0\circ d\varphi)$ is a genuine contactomorphism. This concludes the argument.
\end{proof}

Finally, we treat the parametric case. The following is a consequence of Theorem \ref{thm:EliashbergOT} and Lemma \ref{lem:Gray}. The case of the complement of an overtwisted disk in the $3$-sphere was treated by Dymara in \cite{Dymara:ContRelOT}.

\begin{proposition}\label{prop:ContactomorphismsOTAtInfinity}
    Let $(M,\xi)$ be a contact $3$-manifold overtwisted at infinity. Then, the inclusion $\Cont(M,\xi)\hookrightarrow \FCont(M,\xi)$ is a weak homotopy equivalence.
\end{proposition} 
\begin{proof}
Consider the commuting diagram 
    \begin{displaymath} 
    \xymatrix@M=10pt{
    \Cont_0(M,\xi)\  \ar@{^{(}->}[d]\ar@{^{(}->}[r]  & \Diff_0(M) \ar[r] \ar@{^{(}->}[d] & \C(M;\xi)   \ar@{^{(}->}[d] \\
    \FCont_0(M,\xi) \ar@{^{(}->}[r] & \FDiff_0(M) \ar[r] &  \FC(M;\xi) }
  \end{displaymath}
  in which the rows are the Gray fibrations $G$ and $FG$ from Lemma \ref{lem:Gray} and the columns are the natural inclusions. It follows from Theorem \ref{thm:EliashbergOT} and the Five Lemma that the inclusion $\Cont_0(M,\xi)\hookrightarrow \FCont_0(M,\xi)$ is a weak homotopy equivalence. The result follows now by Lemma \ref{lem:Pi0Cont}
\end{proof}

\section{Strongly overtwisted surfaces}\label{sec:StronglyOvertwistedSurfaces}

In this Section, we prove the $h$-principle for strongly overtwisted surfaces: Theorem \ref{thm:H-PrincipleOvertwistedSurfaces}. To state this result we introduce the notion of germs of (formal) isocontact embeddings near a strongly overtwisted surface. Then, we review the Microfibration Trick (Theorem \ref{thm:microfibration}) introduced in \cite{FMP:Tight} that coupled with Theorems \ref{thm:EliashbergOT} and \ref{thm:OTatInfinity} imply the aforementioned $h$-principle.

\subsection{Formal $\F$-embeddings}\label{subsec:Germs}

Consider an $\R$-invariant contact structure $(\Sigma\times\R,\eta)$, where $\Sigma$ is a compact, connected and oriented surface. In particular, $\Sigma=\Sigma\times\{0\}$ is convex for $\eta$. We will denote by $\mathcal{F}$ the singular foliation determined by $T\Sigma \cap \eta$ which is known as \em characteristic foliation \em of $\Sigma$. The contact structure $\eta$ is only determined by $\F$ up to isotopy. However, from now on, we will write $\xi_{\F}=\eta$ and fix this contact structure for the remaining of the Section. 

Given a contact $3$-manifold we will denote by $\Emb^{\F}(\Sigma,(M,\xi))$ the space of embeddings of $\Sigma\rightarrow (M,\xi)$ which induce the characteristic foliation $\F$ on $\Sigma$. In the case that $\partial \Sigma\neq \emptyset$ is non-empty we will assume that all our embeddings coincide with a pre-assigned embedding near $\partial \Sigma$. We will refer to these embeddings as $\F$-embeddings. To define the formal counterpart of an $\F$-embedding it will suitable for our purposes to speak about germs of isocontact embeddings near a surface. 

Let us first state the general abstract definition. A formal embedding of $\Sigma\times(-\varepsilon,\varepsilon)$ into a $3$-manifold $M$ is a pair $(E,F_s)$ such that 
\begin{itemize}
    \item [(i)] $E:\Sigma\times(-\varepsilon,\varepsilon)\rightarrow M$ is an embedding, 
    \item [(ii)] $F_s:T(\Sigma\times(-\varepsilon,\varepsilon))\rightarrow TM$, $s\in[0,1]$, is a homotopy of bundle isomorphisms covering $E$ such that $F_0=dE$. 
\end{itemize}

We say that two formal embeddings $(E,F_S):\Sigma\times(-\varepsilon,\varepsilon)\rightarrow M$ and $(\hat{E},\hat{F}_s):\Sigma\times(-\hat{\varepsilon},\hat{\varepsilon})\rightarrow M$ are \em equivalent \em if there exists $0<\delta<\min(\varepsilon,\hat{\varepsilon})$ such that $(E,F_s)_{|\Sigma\times(-\delta,\delta)}=(\hat{E},\hat{F}_s)_{|\Sigma\times(-\delta,\delta)}$. Notice that, in particular, $(E,F_s)_{|\Sigma\times\{0\}}=(\hat{E},\hat{F}_s)_{|\Sigma\times\{0\}}$. We will refer to the equivalence class of $(E,F_s)$ under this relation as a \em germ \em of formal embedding over $\Sigma$. Making an abuse of notation we will write $(E,F_s)$ to denote the equivalence class defined by $(E,F_s)$. The space of such germs of formal embeddings is denoted by $\FEmb^{\Germ}(\Sigma,M)$. As customary, in the case that $\partial \Sigma\neq \emptyset$ we require all our germs to coincide near $\partial \Sigma\times\{0\}\subseteq \Sigma\times(-\varepsilon,\varepsilon)$ with some preassigned (germ of) embedding that will be clear from the context. 

\begin{remark}
    We treat spaces of germs as quasi-topological spaces in the same way that spaces of germs of sections are treated by Gromov \cite[pp.35-36]{Gromov:Partial}. Importantly for us, for every compact parameter space $K$, a continuous map $K\rightarrow \FEmb^{\Germ}(\Sigma,M)$ is, by the definition of quasi-topology, represented by a continuous family $K\rightarrow \FEmb(\Sigma\times(-\varepsilon,\varepsilon),M)$ of formal embeddings for some $\varepsilon>0$. This allows to use standard homotopy theoretic notions in this context, as weak homotopy equivalences between spaces of germs of embeddings over $\Sigma$.
\end{remark}

We are interested in the spaces of germs of (formal) isocontact embeddings over a surface with characteristic foliation $\F$. The definitions are as follows: 

\begin{definition}
    \begin{itemize} 
    \item [(i)] A \em formal isocontact $\F$-embedding \em into $(M,\xi)$ is a formal embedding $(E,F_s)$ into $M$ such that $F_1(\xi_{\F})=\xi$. An \em isocontact $\F$-embedding \em into $(M,\xi)$ is a formal isocontact $\F$-embedding of the form $E=(E,dE)$.
     \item [(ii)] A \em germ of a formal isocontact $\F$-embedding \em into $(M,\xi)$ is the equivalence class of a formal isocontact embedding $(E,F_s)$. 
    \item [(iii)] A \em germ of an isocontact $\F$-embedding \em into $(M,\xi)$ is the equivalence class of an isocontact embedding $E=(E,dE)$.
    \end{itemize}
\end{definition}

The space of germs of formal isocontact $\F$-embeddings is denoted by $\FEmb^{\F,\Germ}(\Sigma,(M,\xi))$. In the case $\partial \Sigma\neq \emptyset$ we assume that all our (germs of) formal isocontact $\F$-embeddings coincide and are contact near $\partial\Sigma\times \{0\}\subseteq \Sigma \times(-\varepsilon,\varepsilon)$. The subspace of germs of isocontact $\F$-embeddings is denoted by $\Emb^{\F,\Germ}(\Sigma,(M,\xi))$.

Notice that the natural projection 
$ \pr:\Emb^{\F,\Germ}(\Sigma,(M,\xi))\rightarrow \Emb^{\F}(\Sigma,(M,\xi)) $
is a Serre fibration. The fiber is non-empty \cite{Giroux:Convexity} and, in fact, weakly contractible.

\subsection{The Microfibration trick} 

We review the microfibration technique introduced in \cite{FMP:Tight} in the context of germs of embeddings over a surface. 

Fix a distance $d$ on $M$. For $\varepsilon>0$ and every germ of formal embedding $(E,F_s)\in \FEmb^{\Germs}(\Sigma,M)$ define 
\[ \mathcal{U}_{\varepsilon} (E,F_s):=\{(j,G_s)\in\FEmb^{\Germ}(\Sigma,M): \forall p\in \Sigma, d(j(p,0),E(\Sigma\times\{0\}))<\varepsilon \}.\]

\begin{theorem}[Microfibration Trick \cite{FMP:Tight}]\label{thm:microfibration}
    Let $X\subseteq Y\subseteq \FEmb^{\Germ}(\Sigma,M)$ be two subspaces of formal germs. Assume that 
    \begin{itemize}
        \item [(D.P)] Density Property: For every $\varepsilon>0$ and $(E,F_s)\in Y$ the space $\mathcal{U}_{\varepsilon}(E,F_s)\cap X\neq \emptyset$ is non-empty.
        \item [(L.E.P)] Local Equivalence Property: For every $(E,F_s)\in Y$ there exists some $\varepsilon(E,F_s)>0$ such that 
        $$ \mathcal{U}_{\varepsilon}(E,F_s)\cap X\hookrightarrow \mathcal{U}_{\varepsilon}(E,F_s)\cap Y $$
        is a weak homotopy equivalence for every $0<\varepsilon <\varepsilon(E,F_s)$.
    \end{itemize}
    Then, the inclusion $X\hookrightarrow Y$ is a $C^0$-dense weak homotopy equivalence. That is, for every $\varepsilon>0$, compact $CW$-pair $(K,G)$ and continuous map $H:(K,G)\rightarrow (Y,X)$; there exists a homotopy $H_t:(K,G)\rightarrow (Y,X)$, $t\in[0,1]$, such that 
    \begin{itemize}
        \item [(i)] $H_0=H$,
        \item [(ii)] $H_t(g)=H(g)$ for all $(t,g)\in[0,1]\times G$, 
        \item [(iii)] $\image(H_1)\subseteq X$ and 
        \item [(iv)] $H_t(p)\in \mathcal{U}_{\varepsilon}(H(p))$ for all $(t,p)\in[0,1]\times K$.
    \end{itemize}
\end{theorem}

\subsection{The $h$-principle for strongly overtwisted surfaces}

The main result of this Section is the following:

\begin{theorem}\label{thm:H-PrincipleOvertwistedSurfaces}
    Let $\Sigma\subseteq (M,\xi)$ be an embedded convex strongly overtwisted surface with characteristic foliation $\F$. Then, the inclusion 
    $$ i:\Emb^{\F, \Germ}(\Sigma,(M,\xi))\rightarrow \FEmb^{\F,\Germ}(\Sigma,(M,\xi))$$
    is a $C^0$-dense weak homotopy equivalence.
\end{theorem}
\begin{remark}\label{rmk:PropertiesH-Principle}
  The $h$-principle is also \em relative to domain. \em By this we mean that if $A\subseteq \Sigma$ is a closed subset, then every compact family of germs of formal isocontact $\F$-embeddings $(E^k,F_s^k)$ that is isocontact over $\Op(A)\subseteq \Sigma$ can be homotoped into a family of germs of isocontact $\F$-embeddings relative to $\Op(A)$. This follows  from the proof and the fact that Eliashberg's $h$-principle (Theorem \ref{thm:EliashbergOT}) is relative to every compact domain in the overtwisted at infinity case.
\end{remark}
\begin{proof}
    Let $X=\Emb^{\F,\Germ}(\Sigma,(M,\xi))$ and $Y=\FEmb^{\F,\Germ}(\Sigma,(M,\xi))$. After Proposition \ref{prop:LocalProblem} below we may apply Theorem \ref{thm:microfibration} to  conclude. 
\end{proof}

\begin{proposition}\label{prop:LocalProblem}
Let $(\Sigma\times(-1,1), \xi)$ be a contact $3$-manifold such that 
\begin{itemize}
    \item [(i)] $\Sigma\times\{0\}$ is strongly overtwisted,
    \item [(ii)] $(\Op(\partial \Sigma)\times(-1,1),\xi)=(\Op(\partial \Sigma)\times(-1,1),\xi_{\F})$ and
    \item [(iii)] One of the following two conditions is satisfied
\begin{itemize}
    \item [(a)] If $\Sigma$ is closed then $e(\xi)=e(\xi_{\F})$.
    \item [(b)] If $\Sigma$ has non-empty boundary then $\tb(\partial \Sigma,\xi)=\tb(\partial \Sigma, \xi_{\F})$ and $\rot_{\Sigma}(\partial \Sigma,\xi)=\rot_{\Sigma}(\partial\Sigma,\xi_{\F})$.
\end{itemize} 
\end{itemize}

Then, the following holds: 
\begin{itemize}
    \item [(D.P)] The space $\Emb^{\F,\Germ}(\Sigma,(\Sigma\times(-1,1),\xi))$ is non-empty. 
    \item [(L.E.P)] The inclusion $$ \Emb^{\F,\Germ}(\Sigma,(\Sigma\times(-1,1),\xi))\hookrightarrow \FEmb^{\F,\Germ}(\Sigma,(\Sigma\times(-1,1),\xi)) $$
    is a weak homotopy equivalence.
\end{itemize}
\end{proposition}
\begin{proof}
Let $(K,G)$ be a compact CW-pair and $(E^k,F^k_s)\in \FEmb^{\F,\Germ}(\Sigma,(\Sigma\times(-1,1),\xi))$, $k\in K$, be a family of germs of formal isocontact $\F$-embeddings such that 
$(E^k,F^k_s)=(E^k,F^k_0)=(E^k,dE^k)\in\Emb^{\F,\Germ}(\Sigma,(\Sigma\times(-1,1),\xi)$, for all $k\in G$. 
  
The result will follow if we prove that the existence of a homotopy $$(E^{k,t},F^{k,t}_s)\in \FEmb^{\F,\Germ}(\Sigma,(\Sigma\times[-1,1],\xi)), t\in[0,2],$$ such that 
  \begin{itemize}
      \item [(i)] $(E^{k,t},F^{k,t}_s)=(E^k,F^k_s)$ for $(k,t)\in K\times\{0\}\cup G\times[0,2]$,
      \item [(ii)] $(E^{k,2},F^{k,2}_s)=(E^{k,2},dE^{k,2})\in \Emb^{\F,\Germ}(\Sigma,(\Sigma\times(-1,1),\xi))$ for all $k\in K$. 
  \end{itemize}
Indeed, the property (D.P) then will follow by taking $(K,G)=(*,\emptyset)$ since the non-emptyness of $\FEmb^{\F,\Germ}(\Sigma,(\Sigma\times(-1,1),\xi))$ follows from hypothesis (iii). On the other hand, (L.E.P) will follow by taking $(K,G)=(\D^n,\partial \D^n)$ for all $n\geq 1$. 

Fix a family of formal isocontact embeddings representing our family of germs that we still denote as \[(E^k,F^k_s):(\Sigma\times(-\varepsilon,\varepsilon),\xi_{\F})\rightarrow (\Sigma\times(-1,1),\xi).\]

Consider the family of plane fields $\xi^k_s=F^k_s(\xi_{\F})$, $(k,s)\in K\times[0,1]$, which is defined over $\Op(E^k(\Sigma\times\{0\}))$, and satisfies  $\xi^k_0=\xi_{\F}$ and $\xi^k_1=\xi$.

Let $\delta>0$ be small enough and define $U^\delta=(\Op(\partial\Sigma)\times(-1,1)) \cup (\Sigma\times (-1,-1+\delta))\cup (\Sigma\times(1-\delta,1))$. We may extend these plane fields to a family of formal contact structures on $\Sigma\times(-1,1)$, that we denote in the same way, $$\xi^k_s\in \FC(\Sigma\times(-1,1);\xi), (k,s)\in K\times[0,1],$$
such that 
  \begin{itemize}
      \item [(a)] $\xi^k_1=\xi$ is constant for all $k\in K$,
      \item [(b)] $\xi^k_{0|\Op(E^k(\Sigma\times\{0\}))}=\xi_{\F}$ is contact for all $k\in K$, 
      \item [(c)] $\xi^k_{s|U^\delta}=\xi_{|U^\delta}$ is contact and constant for all $(k,s)\in K\times[0,1]$ and 
 \end{itemize}

Since $\Sigma\times\{0\}$ is strongly overtwisted Theorem \ref{thm:OTatInfinity} implies that $(\Sigma\times(-1,1),\xi)$ is overtwisted at infinity. Therefore, we may apply Theorem  \ref{thm:EliashbergOT}, in a relative way, to find a homotopy of formal contact structures $$\xi^{k,t}_s\in \FC(\Sigma\times(-1,1);\xi), (k,t,s)\in K\times[0,1]\times[0,1],$$ such that $\xi^{k,0}_s=\xi^k_s$ for all $(k,s)\in K\times[0,1]$ and
  \begin{itemize}
      \item [(a')] $\xi^{k,t}_1=\xi$ for $(k,t)\in K\times[0,1]$,
      \item [(b')] $\xi^{k,t}_{0|\Op(E^k(\Sigma\times\{0\}))}=\xi_{\F}$ for all $(k,t)\in K\times[0,1]$,
      \item [(c')] $\xi^{k,t}_{s|U^\delta}=\xi_{|U^\delta}$ for $(k,t,s)\in K\times[0,1]\times[0,1]$; 
      \item [(d)] $\xi^{k,1}_s\in \C(\Sigma\times[-1,1];\xi)$ for all $(k,s)\in K\times [0,1]$.
  \end{itemize}
  Use this homotopy of formal contact structures to define $(E^{k,t}, F^{k,t}_s)=(E^k, F^{k,t}_s)$, $t\in[0,1]$, in such a way that $F^{k,t}_s(\xi_{\mathcal{F}})=\xi^{k,t}_s$. 
  
  Since $\xi^{k,1}_s$ is contact Gray stability applies to find a family of flows $\varphi^{k,t}\in \Diff(\Sigma\times(-1,1))$, $(k,t)\in K\times[0,1]$, such that 
  \begin{itemize}
      \item $\varphi^{k,0}=\Id$ for all $k\in K$ and
      \item $\varphi^{k,t}_*(\xi^{k,1}_{1-t})=\xi^{k,1}_1=\xi$ for all $(k,t)\in K\times[0,1]$.
    \end{itemize}
    We extend the homotopy over $t\in[1,2]$ via the expression 
  \[ (E^{k,t},F^{k,t}_{s}) = (\varphi^{k,t-1}\circ E^k, d\varphi^{k,t-1}\circ F^{k,1}_{(2-t)s})\in \FEmb^{\F,\Germs}(\Sigma,(M,\xi)), (k,t)\in K\times[1,2].\] 
  For $t=2$ we have that  $F^{k,2}_s=d\varphi^{k,1} \circ F^{k,1}_0=d\varphi^{k,1}\circ dE^k=d(\varphi^{k,1}\circ E^k).$ This implies that $(E^{k,2},F^{k,2}_s)\in\Emb^{\F,\Germ}(\Sigma,(\Sigma\times[-1,1],\xi))$ as required.
  
\end{proof}

\section{The $h$-principle for strongly overtwisted contact structures}\label{sec:proofs}

In this Section we prove the main results stated in the Introduction. 

\subsection{Proof of the main results}

We start with the proof of Theorem \ref{thm:MainContactomorphisms}.

\begin{proof}[Proof of Theorem \ref{thm:MainContactomorphisms}]
    Let $(\varphi^k,F^k_s)\in\FCont(M,\xi)$, $k\in \D^n$, be a family of formal contactomorphisms that is actually contact over $\partial \D^n$. That is, $(\varphi^k,F^k_s)=(\varphi^k,d\varphi^k)\in \Cont(M,\xi)$ for all $k\in \partial \D^n$. We will prove the existence of a homotopy $(\varphi^{k,t},F^{k,t}_s)\in \FCont(M,\xi)$, $t\in[0,3]$, such that \begin{itemize}
    \item $(\varphi^{k,0},F^{k,0}_s)=(\varphi^k,F^k_s)$ for $k\in \D^n$ and 
    \item $(\varphi^{k,t},F^{k,t}_s)=(\varphi^{k,t},d\varphi^{k,t})$ is contact for $(k,t)\in (\partial \D^n)\times[0,3]\cup \D^n \times\{3\}$.  
\end{itemize}

We build the homotopy in three steps. First, we make the family contact near a strongly overtwisted surface. Fix the inclusion $i:\Sigma\hookrightarrow (M,\xi)$ of a strongly overtwisted surface with characteristic foliation $\F$, which we may assume to be convex \cite{Giroux:Convexity,Honda:Classification}. This inclusion can be extended to an isocontact embedding $I:(\Sigma\times(-\varepsilon,\varepsilon),\xi_{\F})\rightarrow (M,\xi)$ for some $\varepsilon>0$ small enough. The post-composition 
$$ (\varphi^k\circ I, F^k_s\circ dI)\in \FEmb^{\F,\Germ}(\Sigma,(M,\xi)) $$
defines $\D^n$-family of germs of formal isocontact $\F$-embeddings bounded by a $\partial \D^n$-family of germs of isocontact $\F$-embeddings.  By Theorem \ref{thm:H-PrincipleOvertwistedSurfaces} we may find $(\varphi^{k,t},F^{k,t}_s)\in \FCont(M,\xi)$, $t\in[0,1]$, such that
\begin{itemize}
    \item $(\varphi^{k,0},F^{k,0}_s)=(\varphi^k,F^k_s)$ for all $k\in \D^n$.
    \item $(\varphi^{k,t},F^{k,t}_s)=(\varphi^{k,t},d\varphi^{k,t})\in\Cont(M,\xi)$ for all $(k,t)\in(\partial \D^n)\times[0,1]$,
    \item $(\varphi^{k,1},F^{k,1}_s)_{|\Op(i(\Sigma))}=(\varphi^{k,1},d\varphi^{k,1})_{|\Op(i(\Sigma))}$ is contact for all $K\in \D^n$.
\end{itemize}

Now, we fix the strongly overtwisted surface $i(\Sigma)$. There is a well-defined $\D^n$-family of $\F$-embeddings \[\varphi^{k,1}\circ i\in \Emb^{\F}(\Sigma,(M,\xi)), k\in \D^n\]
The contact isotopy extension theorem for surfaces with fixed characteristic foliation, e.g. \cite{GirouxMassot}, implies the existence of a $\D^n\times[0,1]$-family of contactomorphisms $H^{k,t}\in\Cont(M,\xi)$, $(k,t)\in \D^n\times[0,1]$, such that $H^{k,0}=\Id$ and $H^{k,1}\circ \varphi^{k,1}\circ i=i$. Define
\[(\varphi^{k,t},F^{k,t}_s)=(H^{k,t-1}\circ \varphi^{k,1}, dH^{k,t-1}\circ F^{k,t}_s)\in \FCont(M,\xi), (k,t)\in \D^n\times[1,2] \]

Finally, we make the family contact in the complement of $i(\Sigma)$. Since we may arrange that $(\varphi^{k,2}, F^{k,2}_s)_{|\Op(i(\Sigma))}=\Id_{\Op(i(\Sigma))}$, the family $(\varphi^{k,2},F^{k,2}_s)\in \FCont(M,\xi)$ can be seen as a $\D^n$-family of formal contactomorphisms with compact support in the contact manifold $(M\backslash i(\Sigma), \xi)$. Moreover, since $i(\Sigma)$ is strongly overtwisted, the contact manifold $(M\backslash i(\Sigma),\xi)$ is overtwisted at infinity by Theorem \ref{thm:OTatInfinity}. Therefore, by Proposition \ref{prop:ContactomorphismsOTAtInfinity}, there is a homotopy $(\varphi^{k,t},F^{k,t}_s)\in \FCont(M,\xi)$, $(k,t)\in \D^n\times[2,3]$ such that
\begin{itemize}
    \item $(\varphi^{k,t},F^{k,t}_s)_{|\Op(i(\Sigma))}=\Id_{|\Op(i(\Sigma))}$ for $(k,t)\in \D^n\times[2,3]$, 
    \item $(\varphi^{k,t},F^{k,t}_s)=(\varphi^{k,t},d\varphi^{k,t})\in \Cont(M,\xi)$ for $(k,t)\in (\partial \D^n)\times[2,3]\cup \D^n \times\{3\}$.
\end{itemize}
This concludes the argument.
\end{proof}

\begin{proof}[Proof of Theorem \ref{thm:Main}]
    The proof follows from Theorem \ref{thm:MainContactomorphisms} and Lemma \ref{lem:Gray}.
\end{proof}

\section{Applications}\label{sec:Applications}

In this Section we prove Corollaries \ref{cor:OvertwistedDisks} and \ref{cor:StronglyLooseLegendrians} stated in the Introduction. 

\subsection{Proof of Corollary \ref{cor:OvertwistedDisks}}

We start with the proof of (i). Let $\D^2_{\OT,1},\D^2_{\OT,2}\subseteq (M,\xi)$ be two overtwisted disks in a strongly overtwisted contact $3$-manifold. To prove that both disks are contact isotopic it is enough to find a homotopy of overtwisted disks $\tilde{\D}^2_{\OT,t}$, $t\in[0,1]$, such that $\tilde{\D}^2_{\OT,0}=\D^2_{\OT,1}$ and $(M\backslash(\tilde{\D}^2_{\OT,1}\cup \D^2_{\OT,2}),\xi)$ is overtwisted \cite{Dymara:ContRelOT}.

Consider two smoothly isotopic strongly overtwisted surfaces $\Sigma_1,\Sigma_2\subseteq (M,\xi)$ so that $\Sigma_i\cap \D^2_{\OT,i}=\emptyset$ for $i\in\{1,2\}$. Fix $\varepsilon>0$ such that $U_{\varepsilon}(\Sigma_2)\cap \D^2_{\OT,2}=\emptyset$, here $U_\varepsilon(\Sigma_2)$ denotes an $\varepsilon$-neighborhood of $\Sigma_2$. By Theorem \ref{thm:H-PrincipleOvertwistedSurfaces} there is a contact isotopy $\varphi_t\in\Cont(M,\xi)$, $t\in[0,1]$, such that $\varphi_0=\Id$ and $\varphi_1(\Sigma_1)\subseteq U_{\varepsilon}(\Sigma_2)$. Then, $\tilde{\D}^2_{\OT,t}=\varphi_t(\D^2_{\OT,1})$, $t\in[0,1]$, is the required homotopy of overtwisted disks since $\varphi_1(\Sigma_1)\subseteq M\backslash (\tilde{\D}^2_{\OT,1}\cup \D^2_{\OT,2})$ is strongly overtwisted.

Let us prove (ii). Consider the commuting diagram 

\begin{displaymath} 
    \xymatrix@M=10pt{
    \Cont(M,\xi;\rel \D^2_{\OT})\  \ar@{^{(}->}[d]\ar@{^{(}->}[r]  & \Cont(M,\xi) \ar[r] \ar@{^{(}->}[d] & \Emb^{\OT}(\D^2,(M,\xi))   \ar[d] \\
    \Cont(M,\xi;\rel p) \ar@{^{(}->}[r] & \Cont(M,\xi) \ar[r] &  \CFr(M;\xi) }
\end{displaymath}

in which the rows are fibrations induced by post-composition maps. The first two columns are inclusions and the column at the right is given by the evaluation of the $1$-jet at the origin. Here, the subgroup $\Cont(M,\xi;\rel p)$ is composed by those contactomorphisms that fix $p$ and $\xi_p$. Theorem \ref{thm:MainContactomorphisms} implies that the left column is a weak homotopy equivalence so the result follows from the Five Lemma.  \qed

\subsection{Proof of Corollary \ref{cor:StronglyLooseLegendrians}} 
 
 We prove (i). Since both $L_0$ and $L_1$ are loose we may find overtwisted disks $\D^2_{\OT,i}\in M\backslash L_i$ for $i\in\{0,1\}$. Corollary \ref{cor:OvertwistedDisks} implies that there is a contact isotopy $\varphi_t\in\Cont(M,\xi)$, $t\in[0,1]$, such that $\varphi_0=\Id$ and $\varphi_1(\D^2_{\OT,0})=\D^2_{\OT,1}$. In particular, $L_0$ is Legendrian isotopic to the Legendrian $\tilde{L}_0=\varphi_1(L_0)$ which is in the complement of $\D^2_{\OT,1}$. Every formal isotopy between $\tilde{L}_0$ and $L_1$ can be assumed to avoid $\D^2_{\OT,1}$. Hence, the result follows from the $h$-principle for Legendrians which avoid a fixed overtwisted disk \cite{Dymara:LegendriansOT}, see also \cite[Theorem 7.19]{EliashbergCielebak}.
    
Finally, we explain (ii). Let $\gamma^k$, $k\in \D^n$, be a family of formal Legendrians that is Legendrian for $k\in \partial \D^n$. Fix the inclusion of a convex strongly overtwisted surface $i:\Sigma\hookrightarrow (M,\xi)$ such that $i(\Sigma)\cap\image(\gamma^0)=\emptyset$. After Giroux realization/flexibility and tightness criterion \cite{Giroux:Convexity,Giroux:Tightness}; we can find a disk $\D^2\subseteq \Sigma$ such that $i_{|\D^2}:\D^2\rightarrow(M,\xi)$ is an overtwisted disk. 
    
Let $\F$ be the characteristic foliation on $\Sigma$ determined by $di(T\Sigma)\cap \xi$. Since the parameter space is a disk, Theorem \ref{thm:H-PrincipleOvertwistedSurfaces} allows to find a family of germs of isocontact $\F$-embeddings $I^k:(\Sigma\times(-\varepsilon,\varepsilon),\xi_{\F})\rightarrow (M,\xi)$ such that $\image(I^k)\cap \image (\gamma^k)=\emptyset$ and $I^0=_{|\Sigma\times\{0\}}=i$.

In particular, $i^k_{|\D^2\times\{0\}}$, $k\in \D^n$, is an overtwisted disk embedding in the complement of $\gamma^k$ so the result follows from the parametric $h$-principle for Legendrians with an overtwisted disk in the complement \cite{CardonaPresas}. \qed

\bibliographystyle{plain}
\bibliography{main}

\begin{thebibliography}{10}

\bibitem{Avdek}
Russell Avdek.
\newblock An algebraic generalization of giroux's criterion, 2023.

\bibitem{Bennequin}
Daniel Bennequin.
\newblock Entrelacements et \'{e}quations de {P}faff.
\newblock In {\em Third {S}chnepfenried geometry conference, {V}ol. 1 ({S}chnepfenried, 1982)}, volume 107-108 of {\em Ast\'{e}risque}, pages 87--161. Soc. Math. France, Paris, 1983.

\bibitem{BEM}
Matthew~Strom Borman, Yakov Eliashberg, and Emmy Murphy.
\newblock Existence and classification of overtwisted contact structures in all dimensions.
\newblock {\em Acta Math.}, 215(2):281--361, 2015.

\bibitem{CardonaPresas}
Robert Cardona and Francisco Presas.
\newblock An $h$-principle for embeddings transverse to a contact structure, 2022.

\bibitem{CasalsPancholiPresas}
Roger Casals, Dishant~M. Pancholi, and Francisco Presas.
\newblock The {L}egendrian {W}hitney trick.
\newblock {\em Geom. Topol.}, 25(6):3229--3256, 2021.

\bibitem{EliashbergCielebak}
Kai Cieliebak and Yakov Eliashberg.
\newblock {\em From {S}tein to {W}einstein and back}, volume~59 of {\em American Mathematical Society Colloquium Publications}.
\newblock American Mathematical Society, Providence, RI, 2012.
\newblock Symplectic geometry of affine complex manifolds.

\bibitem{Dymara:LegendriansOT}
Katarzyna Dymara.
\newblock Legendrian knots in overtwisted contact structures on {$S^3$}.
\newblock {\em Ann. Global Anal. Geom.}, 19(3):293--305, 2001.

\bibitem{Dymara:ContRelOT}
Katarzyna Dymara.
\newblock The group of contactomorphisms of the sphere fixing an overtwisted disk, 2005.

\bibitem{Eliashberg:OT}
Y.~Eliashberg.
\newblock Classification of overtwisted contact structures on {$3$}-manifolds.
\newblock {\em Invent. Math.}, 98(3):623--637, 1989.

\bibitem{EliashbergMishachev:Book}
Y.~Eliashberg and N.~Mishachev.
\newblock {\em Introduction to the {$h$}-principle}, volume~48 of {\em Graduate Studies in Mathematics}.
\newblock American Mathematical Society, Providence, RI, 2002.

\bibitem{Eliashberg:Twenty}
Yakov Eliashberg.
\newblock Contact {$3$}-manifolds twenty years since {J}. {M}artinet's work.
\newblock {\em Ann. Inst. Fourier (Grenoble)}, 42(1-2):165--192, 1992.

\bibitem{EliashbergFraser}
Yakov Eliashberg and Maia Fraser.
\newblock Topologically trivial {L}egendrian knots.
\newblock {\em J. Symplectic Geom.}, 7(2):77--127, 2009.

\bibitem{EliashbergMishachev:Tight}
Yakov Eliashberg and Nikolai Mishachev.
\newblock The space of tight contact structures on $\mathbb{R}^3$ is contractible, 2021.

\bibitem{Etnyre:LegendriansOT}
John~B. Etnyre.
\newblock On knots in overtwisted contact structures.
\newblock {\em Quantum Topol.}, 4(3):229--264, 2013.

\bibitem{ContactProblemList}
John~B. Etnyre and Lenhard~L. Ng.
\newblock Problems in low dimensional contact topology.
\newblock In {\em Topology and geometry of manifolds ({A}thens, {GA}, 2001)}, volume~71 of {\em Proc. Sympos. Pure Math.}, pages 337--357. Amer. Math. Soc., Providence, RI, 2003.

\bibitem{FMP:Tight}
Eduardo Fernández, Javier Martínez-Aguinaga, and Francisco Presas.
\newblock The homotopy type of the contactomorphism groups of tight contact $3$-manifolds, part i, 2022.

\bibitem{FMin:Cables}
Eduardo Fernández and Hyunki Min.
\newblock Spaces of legendrian cables and seifert fibered links, 2023.

\bibitem{FM:Dehn}
Eduardo Fernández and Juan Muñoz-Echániz.
\newblock Exotic dehn twists on sums of two contact 3-manifolds, 2022.

\bibitem{Giroux:Convexity}
Emmanuel Giroux.
\newblock Convexit\'{e} en topologie de contact.
\newblock {\em Comment. Math. Helv.}, 66(4):637--677, 1991.

\bibitem{Giroux:Tightness}
Emmanuel Giroux.
\newblock Structures de contact sur les vari\'{e}t\'{e}s fibr\'{e}es en cercles audessus d'une surface.
\newblock {\em Comment. Math. Helv.}, 76(2):218--262, 2001.

\bibitem{GirouxMassot}
Emmanuel Giroux and Patrick Massot.
\newblock On the contact mapping class group of {L}egendrian circle bundles.
\newblock {\em Compos. Math.}, 153(2):294--312, 2017.

\bibitem{Gromov:Partial}
Mikhael Gromov.
\newblock {\em Partial differential relations}, volume~9 of {\em Ergebnisse der Mathematik und ihrer Grenzgebiete (3) [Results in Mathematics and Related Areas (3)]}.
\newblock Springer-Verlag, Berlin, 1986.

\bibitem{Honda:Classification}
Ko~Honda.
\newblock On the classification of tight contact structures. {I}.
\newblock {\em Geom. Topol.}, 4:309--368, 2000.

\bibitem{HondaHuang}
Ko~Honda and Yang Huang.
\newblock Convex hypersurface theory in contact topology, 2023.

\bibitem{Murphy:Loose}
Emmy Murphy.
\newblock Loose legendrian embeddings in high dimensional contact manifolds, 2019.

\bibitem{Vogel:OTDisk}
Thomas Vogel.
\newblock Non-loose unknots, overtwisted discs, and the contact mapping class group of {$S^3$}.
\newblock {\em Geom. Funct. Anal.}, 28(1):228--288, 2018.

\end{thebibliography}
\end{document}